\pdfoutput=1
\documentclass{amsart}
\usepackage{verbatim,amsfonts}
\usepackage[mathscr]{eucal}
\usepackage{graphicx}

\theoremstyle{plain}
\newtheorem{Thm}{Theorem}
\newtheorem{Cor}{Corollary}

\newtheorem{Lem}{Lemma}
\newtheorem{Prop}{Proposition}

\theoremstyle{definition}
\newtheorem{Deff}{Definition}

\newcommand{\mat}[4]{\left( \begin{array}{cc}
{#1} & {#2} \\
{#3} & {#4} \end{array} \right) }

\newcommand{\threesurf}{\ensuremath{\Gamma^3 \backslash {\mathcal H}} }

\begin{document}


\title{McShane's Identity,  Using  Elliptic Elements}
\author{Thomas A. Schmidt}
\address{Oregon State University\\ Corvallis, OR 97331}
\email{toms@math.orst.edu}
\author{Mark Sheingorn}
\address{CUNY - Baruch College \\
   New York, NY 10010} 
\email{marksh@alum.dartmouth.org}
\date{18 February 2008}

\begin{abstract} We  introduce a new method to establish McShane's Identity.  Elliptic elements of order two in the Fuchsian group uniformizing the quotient of a fixed once-punctured hyperbolic  torus act so as to  exclude  points as being highest points of geodesics.   The highest points of simple closed geodesics are already given as the appropriate complement of the regions excluded by  those elements of order two that factor hyperbolic elements whose axis projects to be simple.     The widths of the intersection with an appropriate horocycle of the excluded regions  sum to give McShane's value of $1/2$.   The remaining points on the horocycle are highest points of simple open geodesics,  we show that this set has zero Hausdorff dimension.    
\end{abstract}

\maketitle

\section{Introduction}   In his 1991 Ph.D. dissertation,  G. McShane proved the striking identity 

\[
\sum_{\gamma}\, \dfrac{1}{1 + e^{\ell(\gamma)}} = \dfrac{1}{2}\;,
\]
 where the sum is taken over all simple closed geodesics of any fixed hyperbolic once-punctured torus, and $\ell(\gamma)$ is the length of the geodesic.    This has been reproved in various ways:   \cite{Mc3}, \cite{B}, \cite{GSR};  and generalized variously: \cite{Mc},  \cite{Mc2}, \cite{AHS},  \cite{TWZ}.    The identity has had deep applications due to  Mirzakhani \cite{M}  (see also \cite{R}), \cite{M2}, \cite{M3}.

 We give a proof of the original identity that is, in a sense,   intermediate to McShane's original proof and Bowditch's proof by Markoff triples \cite{B}:   it is geometric; but  lengths of geodesics do not enter directly.      Similar to \cite{GSR}, we take a classical approach; ours involves a uniformizing Fuchsian group.
 We avoid  McShane's  invocation of a deep result of   Birman and Series;  in its place, we show directly  that the appropriate complementary set is of Hausdorff dimension zero, and thus certainly of Lebesgue measure zero, as the identity itself requires.   This Cantor set is the set of apexes of simple open geodesics that achieve their height.       
 
 In the following two paragraphs, we sketch the proof. It is related most directly to the singular punctured sphere that is the quotient of the punctured hyperbolic   torus by its elliptic involution.  The simple closed geodesics on the torus and this sphere  are in 1--1 correspondence (and more). On the sphere, each simple closed geodesic bounces between two elliptic fixed points;   thus,  any hyperbolic element of the uniformizing Fuchsian group whose axis projects to the simple closed geodesic can be factored as a product of elliptics.    But, our first lemma shows that any elliptic element of order two  increases radii of circles whose apexes lie within its {\em uplift  region}, bounded by Euclidean hyperbolas,  in the Poincar\'e upper half plane, $\mathcal H$.

    It is well known (from the work of H. Cohn and others) that there is a lowest horocycle (thus, informally, loop about the cusp) on the hyperbolic torus  beyond which no simple geodesic penetrates; this is true as well for the quotient orbifold.   The appropriate tree of simple closed geodesics' elliptic factorizations gives a set of uplift regions (suitably trimmed) that fit together so as to raise all apexes (below the lift of the fundamental horocycle) other than those of simple closed geodesics' highest lifts.    These regions meet the lift of the fundamental horocycle in disjoint intervals, our {\em excision intervals}, indexed by the tree of simple closed geodesics.   One easily shows that the union of the excision intervals lies in a finite interval;  the complement of their union is a Cantor set.  We show that along all but countably many branches of the tree, the limit of the ratios of  excised interval to ambient interval is 1.     The Cantor set thus has Hausdorff dimension zero.    It in particular has Lebesgue measure zero; but,  the length  of each excised interval is a multiple of a corresponding $\dfrac{1}{1 + e^{\ell(\gamma)}}$, and the full interval has length one-half this multiple, McShane's Identity follows.    
   
 \subsection{Further Remarks} 
 
   The Cantor set in our construction is the set of apexes of lifts of those open simple geodesics that have a highest apex lift.   The endpoints of our excision intervals correspond to geodesics that spiral about a simple geodesic, the remaining points correspond to `irrational laminations'.    These facts can easily be verified by using \cite{H} (see especially Proposition 18 there) or the more recent \cite{BZ}.       Whereas our excision intervals lie along the fundamental horocycle,   \cite{Mc3} finds his {\em gaps} along any fixed horocycle  closer to the cusp than the fundamental horocycle.

   Our approach relies in part on replacing Fricke's equation (in trace coordinates on the Teichm\"uller space): $a^{2} + b^{2} + c^{2} = abc$ by an adjusted equation:  $x^{2}+ y^{2}+ z^{2} = a x y z$.    In the modular case of $a=b=c=3$, it was Cohn's \cite{C} recognition that the adjusted equation is the classical Markoff equation that led him to investigate the geodesy of the corresponding once-punctured hyperbolic torus.    
   
   Our approach should  recover results of \cite{TWZ2} in the setting of a hyperbolic torus with  geodesic boundary.   Generalizations to higher genus must be carefully pursued:   \cite{BLS} shows that there are non-simple geodesics whose self-intersections are not caused by parabolic elements; such geodesics must then be low in the corresponding height spectrum 
and thus of highest lifts of apex exterior to all uplift regions.      It would be interesting to generalize our techniques to hyperbolic surfaces with more general conical singularities, see \cite{DN} and \cite{TWZ}.

Finally, we mention that our approach of trees of triples of order two elements is strongly reminiscent of work of L. Yu. Vulakh on the Markoff and Lagrange spectra and their generalizations, see for example \cite{V}. 
  
\subsection{Outline of Paper}   
Section \ref{backgr} provides background material.   In \S \ref{uplift} we define and give basic results on the basic tool, the uplift regions.   In Section \ref{frickeDoms} we normalize by using work of A.~ Schmidt so as to combine our earlier work on triples of elliptic elements with standard results on trees of simple closed geodesics on hyperbolic tori.   
 We finish the proof in \S \ref{finalArg}.

\subsection{Notation}
 We use $X + i Y$ to denote points in $\mathcal H$.    We call a geodesic of $\mathcal H$ with its standard hyperbolic metric an {\em h-line}.  To increase legibility, all h-lines mentioned are non-vertical (in the Euclidean sense) except as explicitly stated.   
 
\subsection{Thanks}    
We thank Y. Cheung for conversation related to this work.   We also thank the referee for suggestions and references. 

\section{Background}\label{backgr}

\subsection{Tori, Simple Closed Geodesics and Automorphisms}\label{BackSCG}
 
     Each hyperbolic once-punctured torus has a Weierstrass involution.   The quotient of this   torus by its involution gives a singular punctured sphere.   Indeed, there is a one-to-one correspondence between the sets of uniformizing groups:  Fuchsian groups of signature $(0; 2,2,2, \infty)$ and Fuchsian groups of signature $(1;\infty)$, see say  \cite{Sch}.    Furthermore, the simple closed geodesics of each such pair are also in one-to-one correspondence, see say \cite{Sh} (here one finds that a simple closed geodesic on a hyperbolic punctured sphere with three elliptic order two 
 singularities  actually bounces back and forth between two of the singularities); indeed, the lengths of corresponding simple closed geodesics are the same: in fact, there is a common element (primitive in each of the groups) whose axis projects to both geodesics.

   We refer the reader to Section 2 of \cite{JM} for a particularly nice discussion of the structure of the graph of generators of a group of signature $(1;\infty)$,  see the left side of Figure \ref{genDualFig}.   Here, vertices represent group elements, up to inverses and conjugation.     Edges connect pairs that generate the group (which is a free group on two elements); one calls either element in such a pair a {\em generating element}. 
  Any generating element has axis projecting to a simple closed geodesic;  its inverse and any conjugate elements give the same curve (up to orientation).   Thus,  our vertices can be seen as corresponding to the simple closed geodesics of the torus uniformized by the group.     As Bowditch  (see especially the discussion on p. 49 of \cite{BMR}) pointed out,  the dual graph is particularly helpful when discussing Fricke triples, see the right side of Figure \ref{genDualFig}.            (Each node of the resulting tree corresponds to a triple of simple closed geodesics such that a triple of corresponding   open simple  ``cusped'' geodesics is mutually disjoint.)

\begin{figure}
\begin{center}
\includegraphics{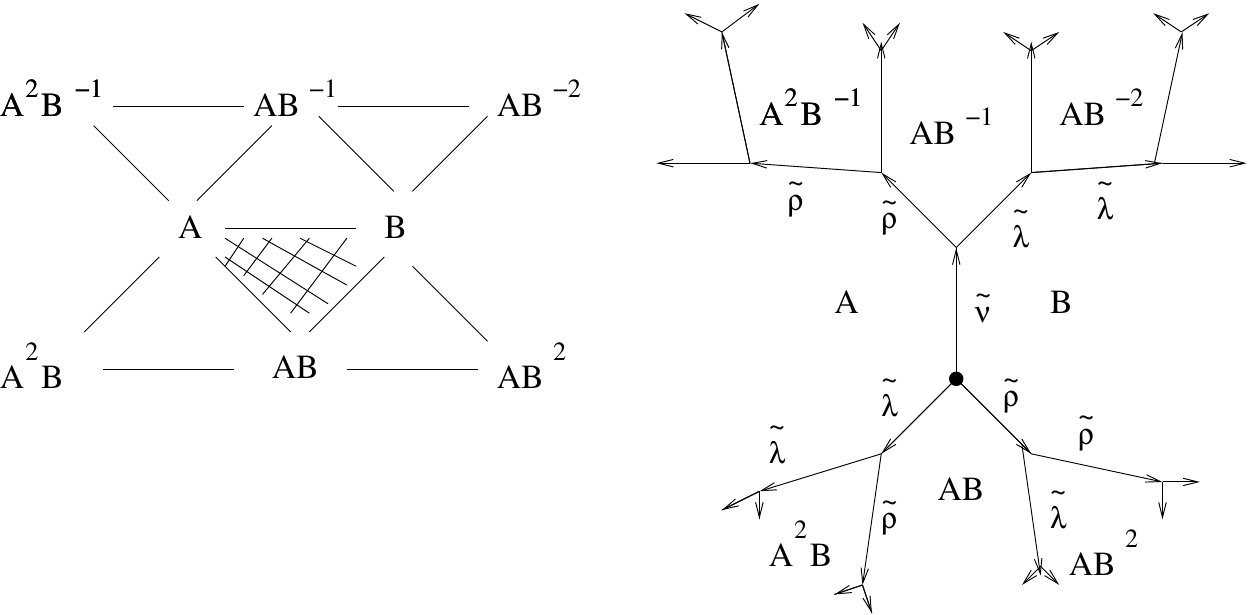}
\caption{Equivalence classes of toral generators and dual graph of triples.}
\label{genDualFig}
\end{center}
\end{figure}

\subsection{Fricke's Equation and Explicit Groups}

We are interested in explicit lifts of simple closed geodesics.  For this, we use a variation of A. Schmidt's application \cite{Sch} of work of Fricke.  Suppose that  positive real $a, b, c$ satisfy  the Fricke equation

\begin{equation} a^2 + b^2 + c^2 = a b c\,, \label{FrickeEq}
\end{equation}
the elements 
\[T_{0} :=  \mat{0}{-a/c}{c/a}{0}, \; 
T_{1} := \mat{a/c}{*}{b/a}{-a/c}\,,\;
T_{2} := \mat{a-b/c}{* }{1}{-a +b/c}
\]
(of determinant one) generate a group of signature $(0; 2, 2, 2; \infty)$.   
 Note that 
 \[ T_2\cdot T_1\cdot T_0 = S^{a}: z \mapsto z + a\]
  is the fundamental translation of this group. 
A full set of  orbit representatives under the action of the Teichm\"uller group is given when one takes $2 < a \le b \le c < ab/2$; this can be deduced from \cite{Sch}, see also  \cite{W}; we will always assume that our Fricke triples $(a,b,c)$ satisfy this restriction. Note that the
modular case of \threesurf corresponds to
$a=b=c=3$ and in this case $T_{j}$ is the conjugate of $T_{0}$ by the translation $z \mapsto z + j$.    (Note that our $T_{0}$ is not that of \cite{Sch}.)

\subsection{Fixed Point Triples and   Fundamental Domains}

For ease of presentation, in \cite{scgArt} we restricted to the modular case.  However, as we noted, our arguments extend to the full Teichm\"uller case. 

\begin{Prop}\cite{scgArt}\label{pssiFunDomThm}   Let the signature $(0; 2,2,2; \infty)$-orbifold $\mathcal U = \Gamma \backslash \mathcal H$ correspond to the Fricke triple $(a,b,c)$.   Each simple closed geodesic has a highest lift which is the axis of 
$S^{a} E$, where $E \in \Gamma$ is elliptic of order two.   There is a factorization 
of $S^{a} E = GF$ as the product of elliptic elements such that 
a highest lifting segment  of this simple geodesic joins the fixed point $f$ of $F$ 
to the fixed point $g$ of  $G$.     Let $e$ be the fixed point of $E$.   A fundamental domain for $\Gamma$  is  given by the
hexagon of vertices: $\infty$, $e$, $f$, $F(e)$, $g$, $a+e$.   In particular, $\{E, F, G\}$ generates $\Gamma$.
\end{Prop} 

  Given a fixed Fricke triple $(a,b,c)$,  we have the corresponding {\em adjusted Fricke equation}
\begin{equation} x^2 + y^2 + z^2 = a \, x y z\,. \label{AdjFrickeEq}
\end{equation}
Note that when $a=3$ the adjusted Fricke equation is exactly Markoff's equation.  

 Recall that the imaginary part,  $Y$,   of a point $X + i Y$  is its {\em height}.  The factorization $S E = G F$ can be used to show that the hyperbolic $GF$, whose axis projects to a simple closed geodesic on the surface, has trace $a z$, where $1/z$ is the height of the fixed point of $E$.   
 One can show that there is such a factorization for every simple closed geodesic, and since the adjusted Fricke equation is satisfied by the traces of appropriate triples of simple hyperbolic elements, one finds the following result. 
   
\begin{Cor}\cite{scgArt}\label{gotTriples}  Let $E, F, G$ be as above.  Then the fixed points of $E, F, G$ have respective  heights  $1/z, 1/y, 1/x$,  whose inverses  give a triple satisfing the adjusted Fricke Equation,  and with $z = \max\{x, y, z\}$.   Furthermore, the simple closed geodesic that lifts to  the axis of $S^{a}E$ has height $r_{a}(z) = \sqrt{a^{2}/4 - 1/z^{2}}$.    The closed geodesic that lifts to the axis of  $ES^{a}E S^{-a}$ has height $R_{a}(z) = \sqrt{a^{2}/4 + 1/z^{2}}$.
\end{Cor}

\begin{proof}   This follows from Theorems 2 and 3 (and their proofs) of \cite{ssEG}, where we give detailed proofs in the modular case.    The only aspect of the proof given there that does not hold in general is that by use of the  map $w \mapsto -\bar{w}$, in the modular case one can further assume that $y\ge x$.
\end{proof}

\begin{Cor}\cite{scgArt}\label{slideMats}  Let $E, F, G$ be as above.  Then there is a real translation conjugating the triple $E, F, G$ to 
\[
E_{0}= \mat{0}{*}{z}{0}, \;  E_{1}= \mat{x/z}{*}{y}{-x/z},
\; E_{2} = \mat
{ax-y/z}{*}{x}{-ax+y/z}\;.
\]
\end{Cor}
\begin{proof}  This follows  as in the proof of Theorem 2 of \cite{ssEG}. 
\end{proof} 

\section{Basic Tool: Uplift Regions}\label{uplift} 

 The following elementary result is key to our approach.

\begin{Lem}\label{basicHypers}   If $A = \begin{pmatrix} \alpha&\beta\\ \gamma & -\alpha\end{pmatrix}$  is in $\text{SL}(2, \mathbb R)$, then  $A$ increases the height of any h-line  with apex $(X,Y)$ satisfying $|\, (X -\alpha/\gamma)^2 - Y^2\,|< 1/\gamma^2$. 
\end{Lem}

\begin{proof}
 We first note that $T_{0}: z \mapsto -1/z$ takes
$C(c, r)$, the circle of real center $c$ and radius $r$, to
$C(\, -c/(c^2-r^2),  r/\vert\,c^2-r^2\,\vert\,)$.     Thus, this element increases
radii whenever $\vert \,  c^2-r^2\,\vert < 1$.    Now, an h-line of apex $(X_{0}, Y_{0})$ has center $X_{0}$ and radius $Y_{0}$.   Thus,  the element $T_{0}$ increases heights for all h-lines of apex of coordinate $(X, Y)$ with $\vert \,  X^2-Y^2\,\vert < 1$.
      
The fixed point of  $A$ is $w = (\alpha + i)/\gamma$.   Since 
$w \mapsto  \gamma w - \alpha$ on $\mathcal H$  respects relative size, 
$A$ increases radii for h-lines of apex $w=X+i Y$ with $|\, (X-\alpha/\gamma)^2 - Y^2\,|< 1/\gamma^2$.
\end{proof}

\subsection{Uplift Regions Defined}
\begin{Deff}  For $A = \begin{pmatrix} \alpha&\beta\\ \gamma & -\alpha\end{pmatrix}$  in $\text{SL}(2, \mathbb R)$,   the {\em uplift region} of $A$, $\mathcal U(A)$,  is the subset of $(X, Y) \in \mathcal H$ such that  
$|\, (X-\alpha/\gamma)^2 - Y^2\,|< 1/\gamma^2$.   We let $\mathcal U_{-}(A)$ denote the elements of the  uplift region of $A$ with $X<\alpha/\gamma$, and $\mathcal U_{+}(A)$ denote the remaining elements.      Finally, we call 
\begin{itemize} 
\item $\{\, (X, Y)\,|\, (X-\alpha/\gamma)^2 - Y^2 = -1/\gamma^2\}$ the {\em upper boundary} of $\mathcal U(A)$, and 
\item $\{\, (X, Y)\,|\, (X-\alpha/\gamma)^2 - Y^2\,= 1/\gamma^2\}$ the {\em lower boundary} of $\mathcal U(A)\,$.
\end{itemize}
 See Figure \ref{upliftBasicsFig}. 
\end{Deff}
\bigskip 

\begin{figure}
\begin{center}
\includegraphics{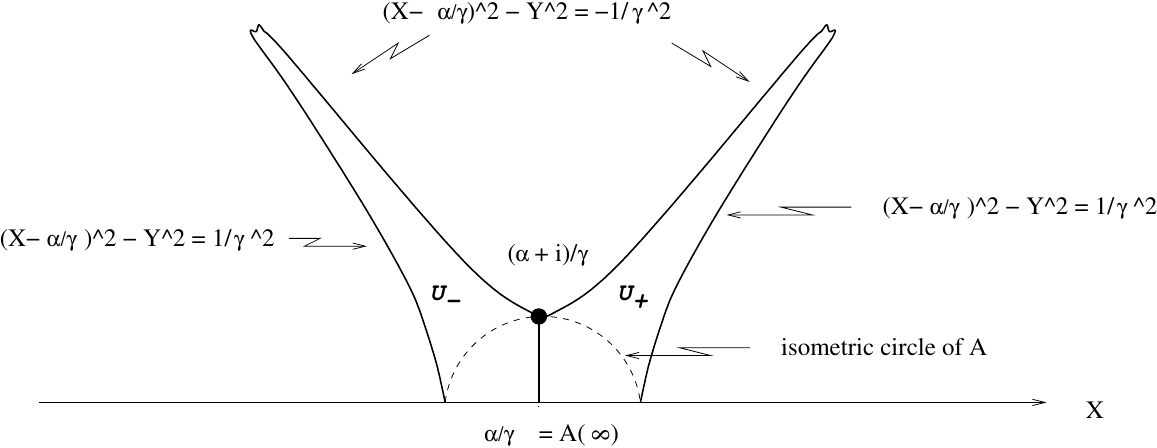}
\caption{Order two $A$ increases heights of h-lines with apex in uplift region $\mathcal U(A) = \mathcal U_{-}(A)  \cup \mathcal U_{+}(A)\;$.}
\label{upliftBasicsFig}
\end{center}
\end{figure}

\bigskip 
 
Recall that the {\em isometric circle} of an element of $\begin{pmatrix} \alpha&\beta\\ \gamma &\delta \end{pmatrix} \in \text{SL}(2, \mathbb R)$ with $\gamma \neq 0$  is the circle of center $-\delta/\gamma$ and radius $1/|\gamma|$.     The isometric circle of an order two elliptic element $A$ (as above) is inscribed in the uplift region,  with points of intersection at  the fixed point of $A$ and at two (ideal) points on the real axis.  The elliptic $A$ acts so as to send its isometric circle to itself, by reflection through the vertical line passing through the fixed point.

\begin{Lem}\label{hitOrPerp}   Suppose that $A = \begin{pmatrix} \alpha&\beta\\ \gamma & -\alpha\end{pmatrix}$  is in $\text{SL}(2, \mathbb R)$ and $\ell$ is an h-line.  Then  $A$ preserves the height of $\ell$ if and only if $\ell$ either passes through the fixed point of $A$, or else $\ell$ meets perpendicularly the isometric circle of $A$.   In the first of these cases, the apex of $\ell$ lies on the upper boundary of $\mathcal U(A)$; in the second, this apex lies on the lower boundary.  
\end{Lem}

\begin{proof}    From Lemma  \ref{basicHypers}, the height is preserved exactly for $\ell$ of apex on the boundary of $\mathcal U(A)$.      To identify the geometry associated to apexes on the components of this boundary, it again suffices to treat the special case of $A = T_{0}$.  This is then a straightforward exercise, easily performed using at most elementary calculus.
\end{proof}

\subsection{Uplift Regions and Translations}\label{upliftRegTransSect}

 Our main application of uplifting is in the setting of triples of elliptic elements of order two whose product is a translation.  

\begin{Prop}\label{meetIsoCirc}   Suppose that $A$, $B$ and $C$ are distinct elliptic elements of order two   such that the product  $ABC$ is a translation.     Then 
  the axis of $AB$ meets perpendicularly the isometric circle of $C$.
\end{Prop} 

\begin{proof} We first show that $C$ fixes the height of the axis of $AB$.  Suppose $w$ lies on this axis, then $BA w$ does as well.   Now, $C( BA w) = (ABC)^{-1}w$ and thus we find that the image of the axis of $AB$ under $C$ is simply a translation of itself.  

The axis of $AB$ passes through the fixed point of each of $A$ and of $B$.   If it also passes through the fixed point of $C$, then each of $A$, $B$ and $C$ send  this axis to itself.    But,  the translation $ABC$ cannot send any h-line to itself.   Therefore, in fact the axis of $AB$ cannot pass through the fixed point of $C$.

Since $C$ fixes the height of the axis of $AB$, but this axis does not pass through the fixed point of $C$,   by Lemma \ref{hitOrPerp} we conclude that the axis of $AB$ meets perpendicularly the isometric circle of $C$.
\end{proof}

\bigskip 

\begin{Deff}  If $A, B \in \text{SL}(2, \mathbb R)$ are elliptic elements of order two, let  $\text{ap}(AB)$ denote the apex of the 
h-line  passing through their fixed points.  (Note that $\text{ap}(AB) = \text{ap}(BA)$, an ambiguity that causes no harm in what follows.) 
\end{Deff}

\begin{figure}
\begin{center}
\includegraphics{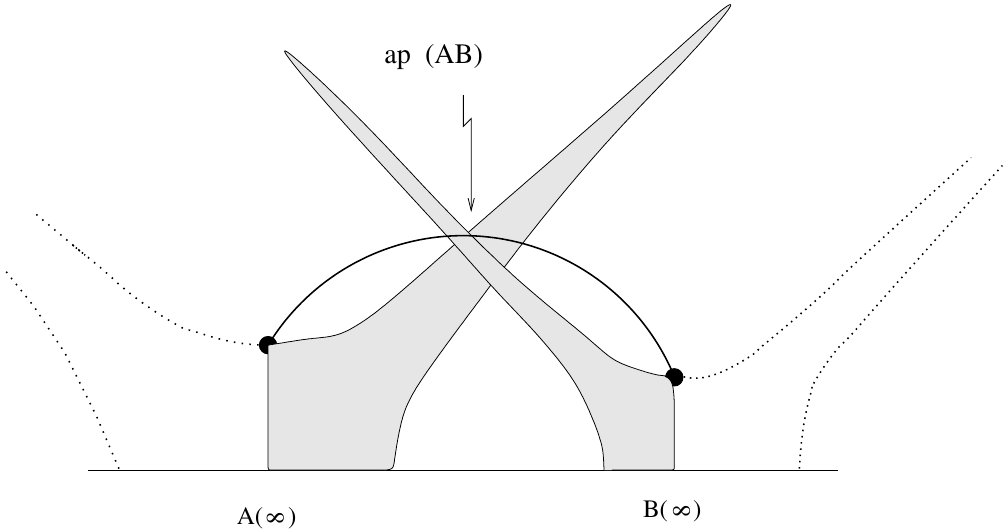}
\caption{Two order two elements:  h-line joining fixed points and bicorn  region.}
\label{fgAdjacentFig}
\end{center}
\end{figure}

For ease of discussion, we will say that uplift regions of two order two elements {\em bifurcate} at a point $p$ if the upper boundaries of these regions intersect at $p$.
  Thus, with $A, B$ as above,  their uplift regions bifurcate at $\text{ap}(AB)$.

\begin{Cor}\label{lowBdry}   Suppose that $A$, $B$ and $C$ are distinct elliptic elements of order two   such that the product  $ABC$ is a translation.     Then $\text{ap}(BC)$ lies on the intersection of the lower boundaries of the uplift regions of $A$ and of $CBABC$.
\end{Cor} 

\begin{proof}   That $\text{ap}(BC)$ lies on the intersection of the lower boundary of the uplift region of $A$ follows by taking inverses and applying  Proposition \ref{meetIsoCirc}  and Lemma \ref{hitOrPerp}.     To show that this apex lies on the lower boundary of the uplift region of $CBABC$,  we can repeat the above, after replacing $A$  
 by  $CBABC$ and $BC$ by its inverse $CB$.
  \end{proof}

\subsection{ Translates of Uplift Regions}

\begin{Deff}  Fix a real number $a >0$.  For $z >2/a$, let $r_{a}(z) = \sqrt{a^{2}/4 - 1/z^{2}}$ and $R_{a}(z) = \sqrt{a^{2}/4 + 1/z^{2}}$.
\end{Deff}   

The final statement of the following Lemma strengthens Corollary \ref{lowBdry} in this setting.

\begin{Lem}\label{plusAlphaQoverlap}  Suppose  $A  \in \text{SL}(2, \mathbb R)$, with  $A = \begin{pmatrix} \alpha&\beta\\ \gamma & -\alpha\end{pmatrix}$ and $S^{a}$ is the translation by $a>0$.  Then  the lower boundaries of $\mathcal U_{+}(A)$ and $\mathcal U_{-}(S^{\alpha} A S^{-\alpha})$ meet at  $(X, Y) = (a/2 + \alpha/\gamma, r_{a}(\gamma)\,)$, while their upper boundaries meet at  $(X, Y) = (a/2 +  \alpha/\gamma, R_{a}(\gamma)\,)\,$.    Furthermore, if $S^{a} A$ is hyperbolic, then 
the apex of its axis lies at the point of intersection of  the lower boundaries of  $\mathcal U_{+}(A)$ and $\mathcal U_{-}(S^{a} A S^{-a})$; similarly,  $\text{ap}(S^{a} A S^{-a}A) = (a/2 +  \alpha/\gamma, R_{a}(\gamma)\,)\,$.
\end{Lem} 

\begin{proof}   This is a trivial computation. (Note that the closures of these uplift regions also meet at two points of $Y$-coordinate $\sqrt{a^{2}/4 + 1/\gamma^{4}}$.)   See Figure ~\ref{aaPlus}. 
\end{proof} 

\begin{figure}
\begin{center}
\includegraphics{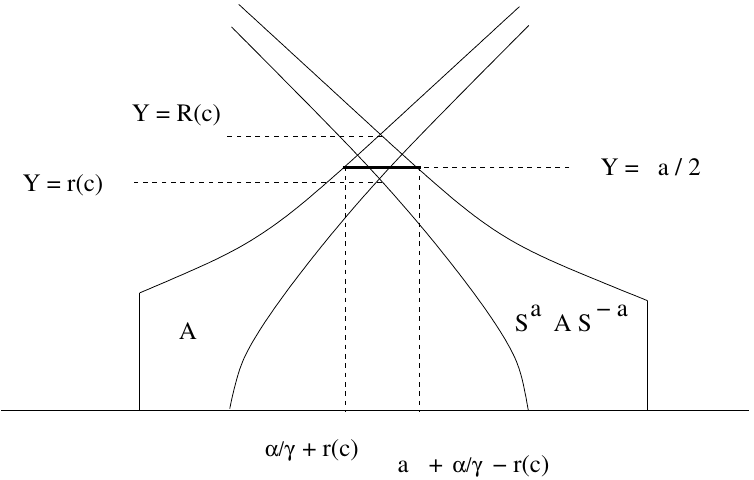}
\caption{Bicorn region for $A$ and $S^{a}AS^{-a}$   meets $Y = a/2$ in     Excision Interval  of $A$.}
\label{aaPlus}
\end{center}
\end{figure}

\begin{Deff}  With notation as above, we call the  intersection of 
 $Y = a/2$ with the union of $\mathcal U_{+}(A)$ and $\mathcal U_{-}(S^{a}A S^{-a})$ the {\em excision interval} of $A$.   
 \end{Deff}

\begin{Lem}\label{howWide}  With notation and hypotheses as above,  the  excision interval of $A$   has width $w_{a}(A) = a - 2  r_{a}(\gamma)$.
\end{Lem} 
 
\begin{proof}   This is also a trivial computation.  See Figure ~\ref{aaPlus}. 
\end{proof}

\section{Fricke-Indexed Fundamental Domains}\label{frickeDoms}

\bigskip

\noindent
{\bf Convention} For the remainder of the paper, unless otherwise stated, we fix a  Fricke triple $(a,b,c)$.     Note that the fundamental translation length is thus $a$.
\bigskip

In \cite{scgArt} and \cite{ssEG}, we showed that \threesurf admits particularly nice fundamental domains indexed by solutions to the original Markoff equation.   Here we summarize this and its direct generalization to the general hyperbolic orbifold of signature $(0; 2,2,2, \infty)$.

\begin{figure}
\begin{center}
 \includegraphics{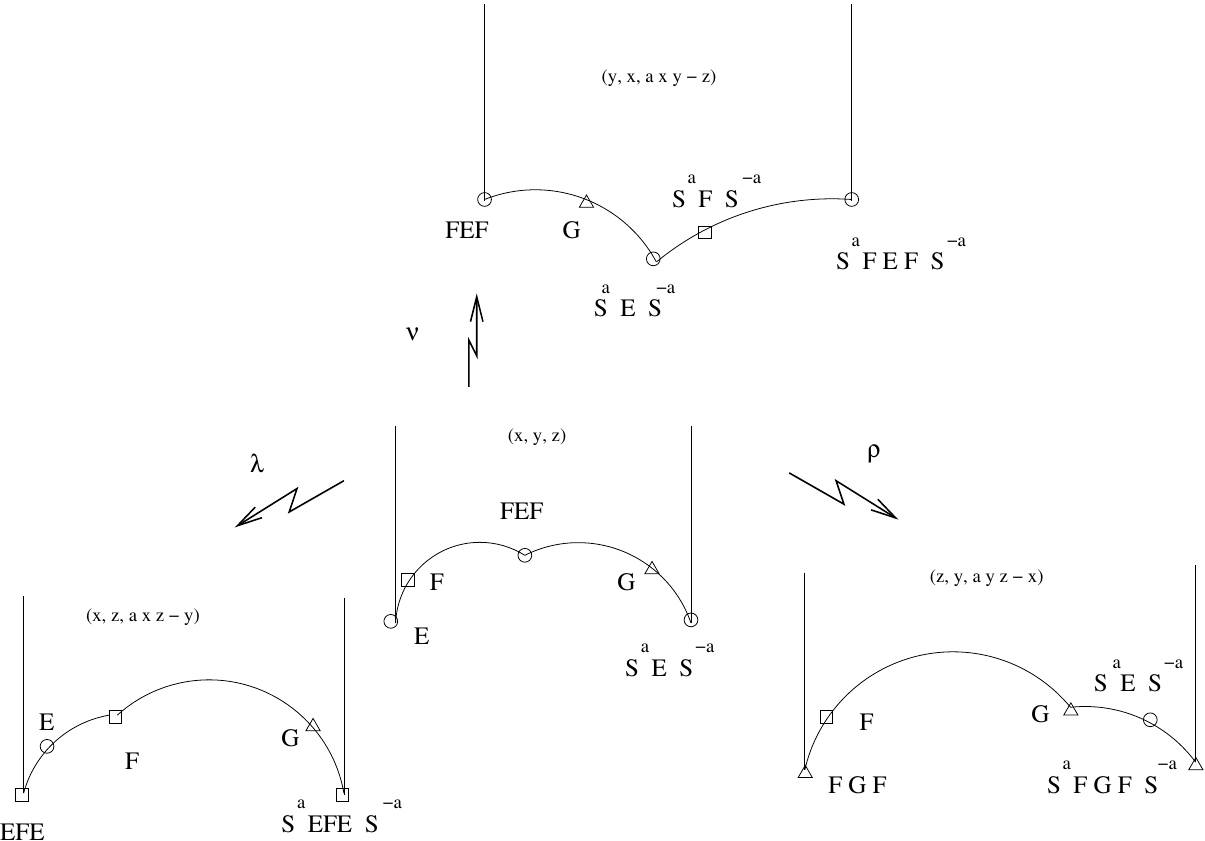}
\caption{Moving through the adjusted Fricke tree.}
\label{frickeTreeFig}
\end{center}
\end{figure}

\newpage 

\subsection{Fundamental Domains, Relating Uplift Regions}\label{relatingUplifts}

\begin{Deff} For any $(E, F, G)$ as in Corollary ~\ref{gotTriples}, we define the following maps to triples of elliptic elements of order two.
\[
\begin{aligned}
\nu: (E,F,G) &\mapsto (FEF, G, S^{a}FS^{-a})\\
\rho: (E,F,G) &\mapsto (FGF, F, S^{a}ES^{-a})\\
\lambda: (E,F,G)  &\mapsto (EFE, E, G)\,.
\end{aligned}
\] 
\end{Deff}

The following is a straightforward computation, compare with Figure \ref{genDualFig}.

\begin{Lem}   Fix some triple $E, F, G$ as above;  let $A = EF$ and $B = FEFG$.     For a homomorphism $\phi: \Gamma \to \Gamma$, let $\tilde \phi$ denote the induced homomorphism on the unique index two subgroup of $\Gamma$ that is of signature $(1;\infty)$ applied to ordered triplets of elements of this subgroup.    Then  
\[
\begin{aligned}
 \tilde \nu(A, B, AB) &= (B, B^{-1}A^{-1}B, A^{-1}B)\\
\tilde \rho(A, B, AB) &= (AB, B, AB^{2})\\
\tilde \lambda(A, B, AB) &= (A, AB, A^{2}B)\,.
\end{aligned}
\]
\end{Lem}

With the above identifications,   the triple of  simple cusped geodesics paired to the triple $(A, B, AB)$  (mentioned in the final sentence of subsection \ref{BackSCG}) 
as seen on $\Gamma \backslash \mathcal H$ is nothing other than the projection of the rays emanating vertically up from the fixed points of $E$, $F$ and $G$.

\begin{Prop}\label{theMoves}     Let  $(E, F, G)$ be as above.  Then each of $\nu(E,F,G), \lambda(E,F,G)$, and $\rho(E,F,G)$ is a generating triple of $\Gamma$.  For each of these triples, the corresponding triple of fixed points gives rise to a solution of the  adjusted Fricke equation, by taking  inverses of  heights, as 
 indicated in Figure ~\ref{frickeTreeFig}.   Furthermore, if $z \ge \max\{x,y\}$ then the analogous inequality holds upon applying either of $\lambda$ or $\rho$.
\end{Prop}

\begin{proof}  That each of these triples also generate $\Gamma$ is easily checked. 

By use of the translated version of the matrices given in Lemma \ref{slideMats}, 
 one  verifies that the triples of multiplicative inverses of the heights of the fixed points of each of the elements involved in $\nu(E,F,G), \lambda(E,F,G)$, and $\rho(E,F,G)$  are as indicated in Figure ~\ref{frickeTreeFig}.     Our hypotheses on $a$  imply  that $FGF$ has the lowest fixed point of the triple $\rho(E, F, G)$ and similarly for $EFE$ and $\lambda(E, F, G)$.
\end{proof}

\begin{Deff}       For $(E, F, G)$ as above, we call the union of  $\mathcal U_{-}(G)$ and  $\mathcal U_{+}(F)$ the associated uplift {\em bicorn region}.    See Figure \ref{fgAdjacentFig}.  
\end{Deff}

\begin{figure}
\begin{center}
\scalebox{.75}{\includegraphics{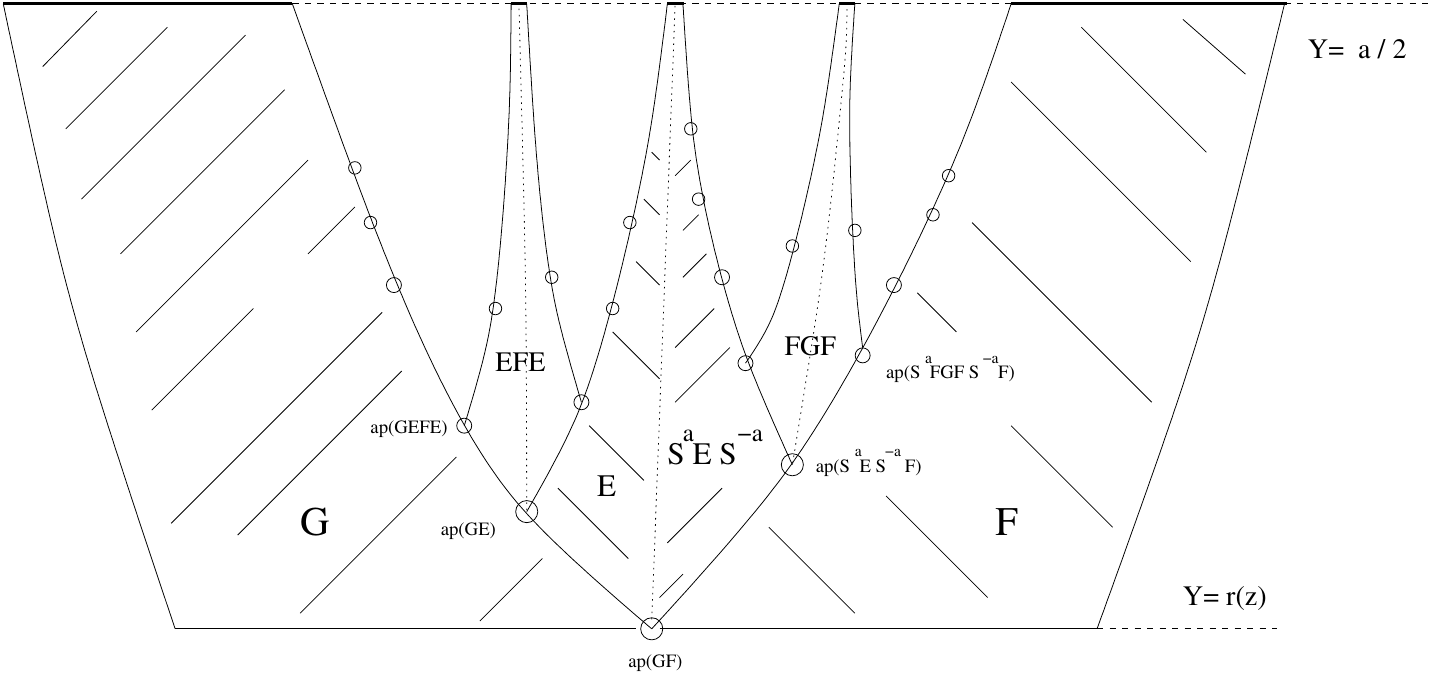}}
\caption{Uplift regions and apexes: Subtree generated by $\lambda$ and $\rho$;  (cross-hatched) trident of Lemma \ref{trident}; some excision intervals.}
\label{bifurcExciseFig}
\end{center}
\end{figure}

   Figure ~\ref{bifurcExciseFig} indicates regions discussed in the following two results. 
 In particular, the cross-hatched region of the figure shows the once-punctured trident formed by the union of the bicorn regions of   $(E, F, G)$, $\rho(E, F, G)$ and $\lambda(E, F, G)$ for one triple $(E, F, G)\,$.   Each non-horizontal dotted curve   indicates the splitting of a region into a  union of some $\mathcal U_{+}(A)$ and $\mathcal U_{-}(S^{a}AS^{-a})\,$. 
  
\begin{Lem}\label{trident}  Let $(E, F, G)$ be as above, with  $z \ge \max\{x,y\}$.      The intersection of $r_{a}(z)\le Y < a/2$ with the union of the bicorn regions of $(E, F, G)$, $\lambda(E, F, G)$ and $\rho(E, F, G)$   has the form of a  once-punctured trident.   The single puncture occurs at  $\text{ap}(GF)$;     the bifurcations of the trident are at $\text{ap}(GE)$ and $\text{ap}(FS^{a}ES^{-a})$.  
\end{Lem}  

\begin{proof}  From the original triple we have that $G F$ has axis projecting to a simple closed geodesic; applying $\rho$, the same is true for  $FS^{a}ES^{-a}$; applying $\lambda$, also for $GE$.   Lemma ~\ref{hitOrPerp} shows that $\text{ap}(GF)$ lies on 
the upper boundary of the uplift regions of $G$ and $F$.    Furthermore, since by construction, the h-line segment joining the fixed points of $F$ and $G$   is a highest lifting segment of the simple geodesic, and these respective fixed points satisfy $\Re(f) < \Re(g)$, we conclude that the union of $\mathcal U_{-}(G)$ and $\mathcal U_{+}(F)$ bifurcates at $\text{ap}(GF)$.    Similar roles are played by $\text{ap}(GE)$ and $\text{ap}(FS^{a}ES^{-a})$.

 By Lemma ~\ref{plusAlphaQoverlap} we have  that  $\text{ap}(GF)$ lies on the lower boundary of both
$\mathcal U_{+}(E)$ and $\mathcal U_{-}(S^{a}ES^{-a})$; furthermore,  
  these regions meet  for $Y$ between the height of $\text{ap}(GF)$ and a value greater than $a/2$.   
 
The result follows by now considering the union. 
 \end{proof} 
 
\begin{Deff}       For $(E, F, G)$ as above,   let $\mathcal T_{\lambda, \rho}(E, F, G)$ denote the tree formed by applying to the triple all finite compositions (including the identity) of $\lambda$  and $\rho$ to $(E, F, G)$, and let $\mathcal U_{\lambda, \rho}(E, F, G)$ denote the  union  of all of the corresponding bicorn regions.
\end{Deff}

\begin{Prop}\label{lamRhoFilling}   Let $(E, F, G)$ be a triple as above, with  $z \ge \max\{x,y\}$.      Then $\mathcal U_{\lambda, \rho}(E, F, G)$
  meets the strip $r_{a}(z) < Y <a/2$ in an infinitely punctured domain bounded by the lower boundary of $\mathcal U_{-}(G)$ and the lower boundary of $\mathcal U_{+}(F)$.   
  \end{Prop} 
 
\begin{proof}    Since $\lambda$ and $\rho$ preserve the property $z \ge \max\{x,y\}$, we can repeatedly invoke the previous lemma.     We thus need only show that successive ``generations'' of uplift triples overlap appropriately.     But, as each of $\lambda$ and $\rho$ retains one of $F$ or $G$ in its original position and promotes either $E$ or $S^{a}ES^{-a}$ to the other, this also easily follows.   
\end{proof}

\subsection{Tree of Triples and Simple Closed Geodesics}

For each Fricke equation, a unique minimum (with respect to the sum of the $x$, $y$ and $z$) solution  exists as \cite{Sch} p. 352 deduces from his Theorem 3.1 (see also  \cite{B});  this thus also holds true for the adjusted Fricke equations  and  \cite{Sch} implies that this minimum solution is given by the multiplicative inverses of the heights of the $T_{i}$.  All solutions (up to cyclic ordering of $x, y, z$) are derived from this minimal solution by sequences of $\nu$, $\lambda$ and $\rho$.   We form a tree completely analogous to that of \cite{Sch}, see his p. 351.

\begin{Deff}  Let $\mathcal T^{\nu}_{\lambda, \rho}$ denote the tree formed by joining $\mathcal T_{ \lambda, \rho}(T_{0}, T_{1}, T_{2})$  to \newline 
$\mathcal T_{\lambda, \rho}(\, \nu(T_{0}, T_{1}, T_{2})\, )$  with an edge (labeled by $\nu$).    The {\em uplift configuration} is the union of the bicorn regions of the nodes of this tree:
\[
\mathcal U^{\nu}_{\lambda, \rho} :=  \mathcal U_{\lambda, \rho}(T_{0}, T_{1}, T_{2}) \cup \mathcal U_{\lambda, \rho}(\, \nu(T_{0}, T_{1}, T_{2})\, )\, .
\]
The {\em normalized uplift configuration} is given by   replacing $\mathcal U_{+}(T_{2})$ in $\mathcal U^{\nu}_{\lambda, \rho}$ 
with its horizontal translation  by $-a$.
\end{Deff}   

\begin{Deff}\label{simplePrimeSet}   Let $\mathcal C$ denote the set of all simple closed geodesics on $\Gamma\backslash \mathcal H$.
\end{Deff} 


 Figure ~\ref{bifurcExciseFig}  indicates some of the geometry of the following result.  
 
\begin{Thm}\label{allScg}    Fix an adjusted Fricke equation.  The normalized uplift configuration  meets the strip $r_{a}(1) \le Y <a/2$ in an infinitely punctured (half-open) domain.   Let $\mathcal P$ be the set of these punctures.    Then $\mathcal P$ is in one-to-one correspondence with $\mathcal C$: each $p \in \mathcal P$  is the apex of a highest lift of some element of $\mathcal C$ and each element of $\mathcal C$ has a highest lift with apex in $\mathcal P$.
\end{Thm} 

\begin{proof}  For ease of notation, let $(x,y,z)$ denote the solution to the adjusted Fricke equation associated to $(T_{0}, T_{1}, T_{2})$.  Since this is a minimal solution,    $\nu$ sends $(T_{0}, T_{1}, T_{2})$ to a triple whose lowest fixed point is given by its  $E$-entry, $T_{1}T_{0}T_{1}$.   Let $w$ be the corresponding entry in the resulting solution to the adjusted Fricke equation.   Now,  Proposition \ref{lamRhoFilling} shows that  $\mathcal U_{\lambda, \rho}(\, \nu(T_{0}, T_{1}, T_{2})\, )$ meets  the horizontal open strip   $r_{a}(w) <Y < a/2$  in an infinitely punctured domain whose right hand boundary is the right hand boundary of $\mathcal U_{+}(T_{2})$ and whose left hand boundary is that of $\mathcal U_{-}( S^{a}T_{1}S^{-a})$.
  But,  by Lemma \ref{plusAlphaQoverlap},  this left hand boundary is contained in $\mathcal U_{+}( T_{1})$ for  $r_{a}(y)<Y<a/2$.     Thus,  $\mathcal U_{\lambda, \rho}(\, \nu(T_{0}, T_{1}, T_{2})\, )$ and $\mathcal U_{\lambda, \rho}(T_{0}, T_{1}, T_{2})$ have non-trivial intersection.   The union, $\mathcal U^{\nu}_{\lambda, \rho}$  thus meets the strip $r_{a}(z) < Y <a/2$ in an infinitely punctured domain.

    The normalization simply replaces $\mathcal U_{+}(T_{2})$ by $\mathcal U_{+}(S^{-a}T_{2}S^{a})$; due to  Lemma \ref{plusAlphaQoverlap}, the intersection with the strip remains an infinitely punctured domain. 

From Lemma \ref{trident},   each of $\mathcal U_{\lambda, \rho}(T_{0}, T_{1}, T_{2})$ and $\mathcal U_{\lambda, \rho}(\, \nu(T_{0}, T_{1}, T_{2})\, )$  contributes elements to $\mathcal P$ that are apexes.   That the lifts in question are highest lifts of simple closed geodesics follows by observing the geometry of the fundamental domains (each of which has a single ideal vertex).    

There are exactly two remaining elements of $\mathcal P$:  one introduced by taking the union of  $\mathcal U_{\lambda, \rho}(T_{0}, T_{1}, T_{2} )$ and $\mathcal U_{\lambda, \rho}(\, \nu(T_{0}, T_{1}, T_{2})\, )$; the second an artifact of our normalization.      The first is the puncture lying on the intersection of the lower boundaries of $\mathcal U_{+}(T_{1})$ and $\mathcal U_{-}(S^{a}T_{1}S^{-a})$.    By  
 Lemma \ref{plusAlphaQoverlap} this is the apex of the axis of $S^{a}T_{1}$.   Similarly, our normalization introduces the puncture given by apex of the axis of $T_{2}S^{-a}$.     Now,  $\text{ap}(S^{a}T_{1})$ lies on $Y = r_{a}(y)$ which is lies above or is the line $Y = r_{a}(x)$, the horizontal line upon which lies $\text{ap}(T_{2}S^{-a})$; this as $x=1$ is the minimum of the triple of inverse of heights of the fixed points of $T_{0}$, $T_{1}$ and $T_{2}$.   These two apexes lie on highest lifts of   the simple closed geodesics  (seen by using conjugation and taking inverses) that are the projections of the axes of $B = T_{1}T_{0}T_{1}T_{2}$ and 
$A = T_{0}T_{1}$, respectively.

 Finally, by the discussion  in the previous subsection, replacing each node $(E, F, G)$ of $\mathcal T^{\nu}_{\lambda, \rho}$ by   the triple  $(EF, FEFG, FG)$ gives the tree of all triples of associated simple closed geodesics on the canonical hyperbolic once punctured torus double (ramified) covering $\Gamma \backslash \mathcal H$.   But,  the bicorn region at each node of $\mathcal T^{\nu}_{\lambda, \rho}$ gives the element $\text{ap}(FG) \in \mathcal P$.   That is, associated to each node of this tree, is the apex of a highest lift of the simple closed geodesic whose face in the dual graph has (directed) edges labeled by $\tilde \lambda$ and $\tilde \rho$ emanating from the given node $(EF, FEFG, FG)$.    Since our initial node is $(T_{0}, T_{1}, T_{2})$,  we conclude that these elements of $\mathcal P$ are the apexes of highest lifts for all elements of $\mathcal C$ other than the projection of the axes of $A$ and $B$ (as defined above).    But, we have already seen that the remaining points of $\mathcal P$ account exactly for these two simple closed geodesics.     We thus conclude that $\mathcal P$ is exactly in one-to-one equivalence with $\mathcal C$, by associating apexes to projections of corresponding h-lines.   
\end{proof}

\subsection{The Line $Y=a/2$} 

\begin{Prop}\label{intervalsAtLine}   The normalized uplift configuration meets the line $Y = a/2$ in the union of disjoint intervals:
\[ \bigsqcup_{T}\;  \mathcal U_{+}(T ) \cup \mathcal U_{-}(S^{a}TS^{-a})\,,\]
where the union is over all order two elements $T$ appearing in the triple for any node of the  tree $\mathcal T^{\nu}_{\lambda, \rho}$.     
\end{Prop} 

\begin{proof}   By Lemma \ref{plusAlphaQoverlap}, each  $\mathcal U_{+}(E) \cup \mathcal U_{-}(S^{a}ES^{-a})$ meets $Y=a/2$ in an interval.   Observing  the action of $\nu$, $\lambda$ and $\rho$, one easily sees that the normalized uplift configuration meets the line $Y = a/2$ in the union of the intervals indexed by the various $E$.      

It thus suffices to show that the various $ \mathcal U_{+}(E ) \cup \mathcal U_{-}(S^{a}TS^{-a})$ meet the line disjointly.    But, we already know that $\mathcal P$ lies below $Y = a/2$;  by Lemma \ref{lamRhoFilling}, $\mathcal P$ contains the set of bifurcation points of the (normalized)  uplift configuration.     Disjointness follows. 
 \end{proof}

\begin{Cor}\label{upBndSumExcInt}  The sum of the $w_{a}(T) = a - \sqrt{a^{2} - 4/z^{2}}$,  indexed over  the order elements $T$   appearing in the triple for any node of the  tree $\mathcal T^{\nu}_{\lambda, \rho}$, is at most $a$. 
\end{Cor}  

\begin{proof}  The  uplift configuration fills in from $\mathcal U_{-}(T_{2})$ to $\mathcal U_{+}(T_{2})$; the closure of the normalized uplift region thus meets the line in a region contained in the interval from the left endpoint of the intersection with $\mathcal U_{+}(S^{-a}T_{2}S^{a})$ to the left endpoint of the intersection with $\mathcal U_{+}(T_{2})$.   This ambient interval is  of length $a$. 
\end{proof}
  
\section{Final Arguments}\label{finalArg} 
 
\subsection{Upper Bound:   Lengths of Excision Intervals, Lengths of Geodesics}\label{lengths}

Recall that the length of a closed geodesic on a hyperbolic surface is $\ell(\gamma) = 2 \ln \epsilon_{\gamma}$, where $\epsilon_{\gamma}$ is the larger solution of  $\epsilon_{\gamma}+ 1/\epsilon_{\gamma} = t$ for $t= |\, t(M)\,|$ the absolute value of the trace of a primitive element whose axis projects to $\gamma$.      From Theorem \ref{allScg},  each simple closed geodesic $\gamma$ is the projection of the h-line of apex some element of $\mathcal P$.   In general, this gives $\gamma$ as the projection of 
the axis of $S^{a}E$ with $(E, F, G)$ a uniquely corresponding node of $\mathcal T^{\nu}_{\lambda, \rho}$; the corresponding  triple $(x,y,z)$ is such that $t = az$.   (As in the proof of Theorem \ref{allScg},  the two simple closed geodesics distinguished as artifice of our indexing are the projections of the axes of  $S^{a}T_{1}$ and $S^{a}T_{2}$, of  trace $t = ab$ and $t = ac$, respectively.) 

 One easily calculates that

\begin{equation}
 \dfrac{1}{1 + e^{\ell(\gamma)}} = \dfrac{w_{a}(E)}{2a}\;.
\end{equation}                                               
Corollary \ref{upBndSumExcInt}  thus  implies that  the sum of all  $\dfrac{1}{1 + e^{\ell(\gamma)}}$ is at most $1/2$.

 \subsection{Lower Bound:  Hausdorff Dimension Zero} 
  We must show that the upper bound given in Corollary \ref{upBndSumExcInt}  is also a lower bound.      To do this, it suffices to show that the Cantor set formed by deleting the union of the excision intervals indexed by the nodes of $\mathcal T^{\nu}_{\lambda, \rho}\,$ has Lebesgue measure zero.    In fact,  an analysis not unlike that in \cite{B} of limits along branches of our tree, reveals that much more is true.    Using  Proposition~\ref{limRatios}  below,  we show the following.  
  
 \begin{Thm}\label{hausdorff}  The complement to the union of the excision intervals is a set of zero Hausdorff dimension.   
\end{Thm} 

\begin{proof}   It suffices to show that $s=0$ is an upper bound for this Hausdorff dimension.

Consider first a  Cantor set constructed iteratively by removing a centered subinterval with fixed ratio of $k\in(0,1)$ from each interval remaining at the $n$-th step.   If the original interval has finite length $L$, then at the $n$-th iteration there are $2^{n}$ intervals each of length $(\,(1-k)/2\,)^{n}L$.    An upper bound for the Hausdorff dimension is then obtained by  finding the unique value of $s$ such that  $\lim_{n\to \infty} 2^{n} (\,(1-k)/2\,)^{ns}L^{s}$ is finite and non-zero --- this is $s = \log 2/\log (2/(1- k)\,)$.

In the case where a Cantor set is formed by removing possibly non-centered subintervals, but with a constant ratio of $k$,  a naive upper bound for the lengths of intervals at the $k$-th iteration is simply  $(1-k)^{n}L$.    The corresponding upper bound on the Hausdorff dimension is $s(k) :=  \log 2/\log (1/(1-k)\,)$.     Note that $s$ tends to zero as $k$ increases to 1.  

 We now turn to our Cantor set.    By Proposition~\ref{limRatios} (below),  for any $\epsilon > 0$ there are at most finitely many nodes of $\mathcal T^{\nu}_{\lambda, \rho}\,$ such that the ratio of the corresponding excision interval to its ambient interval is less than $1- \epsilon$.   Excision of these intervals leads to finitely many subintervals; restricting the excision process to each gives a Cantor set, of Hausdorff dimension at most $s(1-\epsilon)$.     The Hausdorff dimension of their union, our Cantor set, thus has this same upper bound. 
 Letting $\epsilon$ tend to zero, we find that $s=0$ is indeed an upper bound. 
 \end{proof}

 \begin{Prop}\label{limRatios}   Fix a  directed branch beginning at $(T_{0}, T_{1}, T_{2})$ in $\mathcal T^{\nu}_{\lambda, \rho}\,$.   The limit of the  ratio of lengths of the excision interval of $E$ to the ambient interval  (on $Y=a/2$) bounded by $\mathcal   U_{-}(G)$ and $\mathcal U_{+}(F)$ equals $1$ unless the branch eventually ends in an infinite sequence of exactly one of $\lambda$ or $\rho$.  In this purely periodic case, there is an $x$ as above such that the limit is 
 $\dfrac{ 2 \sqrt{a^{2}-4/x^{2} } }{a +\sqrt{a^2  - 4/x^2}} \,$.   
 \end{Prop} 
 
 \begin{proof} 
 
 We delete the excision interval of $E$ from the ambient interval lying between  the excision interval of $S^{-a} G S^a$ and that of $F$.    Thus,  by Lemma \ref{howWide},   we excise an interval of length 
$ a - 2 r_a(z)$     from  one of  length    
$F(\infty) + r_a(y) - ( G(\infty) - r_a(x) )$.   But, using the translated versions of our matrices, given in Corollary \ref{slideMats} on page ~\pageref{slideMats}, we find that this latter interval has length  $x/yz + y/xz - a + r_a(x) + r_a(y)$.    Now  $x/yz + y/xz - a = -z^2/xyz$;  solving Equation \eqref{AdjFrickeEq}, the adjusted Fricke equation,   for $z$,   (with $z$ sufficiently large) allows us to write  $z/(xy) = a/2 + \sqrt{a^{2}/4 - 1/x^{2}- 1/y^2}$.    We are thus to find the limit of 
\begin{equation}\label{ratioExpress}
\dfrac{            a - 2 r_a(z) }{ r_a(x) + r_a(y) - (\,a/2 + \sqrt{a^{2}/4 - 1/x^{2}- 1/y^2}\,)}\;.
\end{equation}
Throughout our proof we use Taylor series approximation of $f(\delta) = \sqrt{s^2 -\delta}$ around $\delta = 0$:   
 \[\sqrt{s^2 - \delta} =    s - \delta/(2 s)  - \sum_{j=2}^{N} \,c_{j}\,  \dfrac{\delta^{j}}{s^{2 j-1}}\, + O(\delta^{N+1})\;,\]
 with $c_{j} = \dfrac{ 1\cdot3 \cdots (2j-3)}{2^{j}\,j!}\, $.   
In particular, the numerator of our ratio is 
\begin{equation}
       a - 2 r_a(z) = \dfrac{2}{a z^{2}} + O(z^{-4})\,.
\end{equation}

  Our denominator is symmetric in $x$ and $y$;  we can and do relabel each pair such that $x \le y$  (we thus no longer demand that $F$ fixes the point whose height is $1/y$). We now treat three cases:  our branch eventually ends in repeating exactly one of $\rho$ or $\lambda$;  it has unbounded blocks of either $\rho$ or $\lambda$; and finally,  it has   bounded blocks of either.       

{\em Eventually repeating $\rho$ or $\lambda$.}   We first treat the case of   the branch eventually repeating in one of $\rho$ or $\lambda$.  Here,  the smallest value of each triple, $x$ (eventually) remains constant, whereas $y$ and $z$ both go to infinity.      Since $z/xy = a/2 + \sqrt{ (r_{a}(x)\,)^{2} -1/y^{2}}$, two term approximation gives
 $a/2 + r_{a}(x) - \frac{1}{2 r_{a}(x) y^{2}} + O(y^{-4})$.   Using two term approximation on $r_{a}(y)$ as well, we find that  the denominator is 
 \[- \frac{1}{a y^{2}} + \frac{1}{2 r_{a}(x) y^{2}} + O(y^{-4})  = \dfrac{a - 2 r_{a}(x)}{ 2 a   r_{a}(x) y^{2}}+ O(y^{-4})\,.\]
 
Finally,  $z/y = \frac{x}{2} (a + 2 r_{a}(x)\,) + O(y^{-2})$, and $(a + 2 r_{a}(x)\,)(a - 2 r_{a}(x)\,)  = 4/x^{2}$.    Hence we find 
 \[
 \begin{aligned} 
\dfrac{a - 2 r_a(z) }{ r_a(x) + r_a(y) - z/xy} &= \dfrac{2/a + O(z^{-2}) }{(z/y)^{2}(\, \frac{ a - 2 r_{a}(x)}{2 a r_{a}(x)} + O(y^{-2})\,)}\\
           &=\dfrac{2/a + O(z^{-2}) }{  \frac{ a + 2 r_{a}(x)}{2 a r_{a}(x)} + O(y^{-2}) }\\
           &=  \dfrac{4 r_{a}(x) + O(z^{-2}) }{ a + 2 r_{a}(x) + O(y^{-2}) }\\\;.
 \end{aligned}          
 \] 
 Taking the limit with $x$ fixed and $y, z$ tending to infinity gives 
\[
\dfrac{4 r_{a}(x)   }{ a + 2 r_{a}(x)} = \dfrac{ 2 \sqrt{a^{2} - 4/x^{2}}}{ a + \sqrt{a^{2}- 4/x^{2}}}\;,
\]
 as claimed.   Note that  these values tend to $1$ as  $x$ itself tends to infinity.   
 
  \bigskip

{\em Unbounded blocks of $\rho$, $\lambda$.}  
 Consider any branch where the number of consecutive nodes of $\rho$ or of $\lambda$ is unbounded.     For each positive integer $N$  there is an infinite set of disjoint blocks of $N$ consecutive applications of $\lambda$ or $\rho$ on the branch.     But, as $N$ increases, we thus find that the ratios of lengths of excised to ambient interval give ever better approximations to the limit ratios along (eventually) constant branches.    Moreover, on our branch we must have that   these corresponding values of $x$ are also (eventually) increasing.   Thus, the limit of ratios along this branch equals  $\lim_{x \to \infty}\,  \dfrac{ 2 \sqrt{a^{2}-4/x^{2} } }{a +\sqrt{a^2  - 4/x^2}}\,$.   But,  this limit equals $1$.

 \bigskip 

{\em Bounded blocks of $\rho$, $\lambda$.}    On any branch not given by eventually repeating $\lambda$ or $\rho$,  each of $x$, $y$ and $z$ goes to infinity.     We again first concentrate on the denominator, using Taylor series with our assumption  that $x\le y$, 

\[
\begin{aligned}
&\sqrt{a^{2}/4 - 1/x^{2} } + \sqrt{a^{2}/4 - 1/y^2} - (\, a/2 + \sqrt{a^{2}/4 - 1/x^{2}- 1/y^2}\,)\\
\\
 & \phantom{ a^{2}/4-}= \sum_{j=2}^{N} (\dfrac{2}{a})^{2j -1}\,c_{j}\,[\, (x^{-2}+y^{-2})^{j} - (x^{-2j}+y^{-2j})  \,]+ O(x^{-2(N+1)})
\end{aligned}
 \]

\bigskip

\[
\begin{aligned}
&\phantom{}= \dfrac{2}{a^{3}x^{2}y^{2}}\cdot \big( \, 1 + \dfrac{3}{a^{2}} (x^{-2}+ y^{-2}) +\cdots +  \dfrac{2^{N-2}c _{N}}{a^{2N -4}} \sum_{k=1}^{N-1}\binom{N}{k} x^{-2(N-k) + 2} y^{-2k + 2} \; \big)    \\
&  
 \;\;\;\;\;\;  \;\;\;\;\;\;    + O(x^{-2(N+1)})      \;. 
 \end{aligned}
\]

\bigskip
 We thus find that our ratio is 
 
 \[
 \dfrac{1 + O(z^{-2})}{ \dfrac{z^{2}}{a^{2}x^{2}y^{2}}\cdot \big( \, 1  +\cdots + \frac{2^{N-2}}{N!}\; \frac{ 1\cdot3 \cdots (2N-3)}{a^{2N -4}} \sum_{k=1}^{N-1}\binom{N}{k} x^{-2(N-k) + 2} y^{-2k + 2} \; \big)  + z^{2}O(x^{-2(N+1)}) }     \;. 
 \]

 Since $\dfrac{z^{2}}{x^{2}y^{2}} = a^{2} +O(x^{-2}+y^{-2})$,  we find that the limit equals $1$ if there is some $N$ such that $z^{2}/x^{2N}$ has a finite limit (on our given branch).        We apply the next Lemma (replacing $N$ here by at worst $2N+4$, since we can assume that $a<x\,$).
 \end{proof}

 \bigskip 
 
 For clarity's sake, we use $(x,y,z)$ as in previous sections, and at each node let $l = \min(x,y)$.

\begin{Lem}   Fix a  directed branch beginning at $(T_{0}, T_{1}, T_{2})$ in $\mathcal T^{\nu}_{\lambda, \rho}\,$.    If the length of blocks of consecutive $\rho$ or $\lambda$ along the branch is bounded by $N$, then at each node (beyond the first change between $\lambda$ and $\rho$), one has  $z <   (a l)^{N+2}$.   
\end{Lem} 
 
\begin{proof}   Recall that  $\rho:(x,y,z) \mapsto (z, y, a y z - x)=:(x',y', z')$ and  $\lambda:(x,y,z) \mapsto (x, z,  a x z - y)=:(x'',y'', z'')$.    Beginning at any node  of corresponding triple $(x,y,z)$,   induction shows  that $y'$ and $x''$ are the minimum of their respective triples.    (The main base case relies on the minimality of the solutions from $(T_{0}, T_{1}, T_{2})\,$; the secondary base case arising from $\nu(T_{0}, T_{1}, T_{2})$ is easily verified.)  

  We always have $z < a xy$.    Thus,  if $l = x$, then  $z'< a^{2}y'^{3}$ and if $l=y$ then $z'' < a^{2} x''^{3}$.     Therefore, under these respective assumptions,   we find $z < a^{2}l^{3}$ holds for this new generation.

Now suppose that $z\le a^{k} l^{n}$.   Then with $l=y$ we have $z'  < a^{k+1} y^{n+1}$.    With $l = x$, we have  
  $z'' < a^{k+1} x^{n+1}$.      That is, under these assumptions, we find    $z \le a^{k+1} l^{n+1}$ holds for this new generation.

In summary,  after each change to either $\rho$ to $\lambda$,  we have $z < a^{2}l^{3}$;  this followed by $n-1$ more (consecutive) applications of the current $\rho$ or $\lambda$    then gives $z < a^{1+n}l^{2+n}$.    The result follows. 
\end{proof}

\end{document}